\documentclass{amsart}
\usepackage{amssymb,amsmath,amsthm}
\usepackage[all]{xy}   

\usepackage{graphicx}
\usepackage{subfig}    

\usepackage{color}

\newtheorem{theorem}{Theorem}[section]
\newtheorem{proposition}[theorem]{Proposition}
\newtheorem{lemma}[theorem]{Lemma}
\newtheorem{corollary}[theorem]{Corollary}
\theoremstyle{definition}
\newtheorem{definition}[theorem]{Definition}
\newtheorem{example}[theorem]{Example}

\theoremstyle{remark}
\newtheorem{remark}[theorem]{Remark}

\newcommand{\newword}[1]{\textbf{#1}}
\newcommand{\C}{\mathbb{C}}
\newcommand{\R}{\mathbb{R}}
\newcommand{\Z}{\mathbb{Z}}
\newcommand{\N}{\mathbb{N}}
\newcommand{\stab}{\mathrm{stab}}
\newcommand{\rist}{\mathrm{rist}}

\newcommand{\includeapd}[1]{\includegraphics[scale=0.833333]{#1}}

\begin{document}

\title{A Thompson group for the basilica}

\author{James Belk}
\address{School of Mathematics and Statistics \\ University of Glasgow \\ Scotland}
\email{jim.belk@glasgow.ac.uk}

\author{Bradley Forrest}
\address{The Richard Stockton College of New Jersey \\ Mathematics Program \\ 101 Vera King Farris Drive \\ Galloway, NJ 08205}
\email{bradley.forrest@stockton.edu}



\begin{abstract}
We describe a Thompson-like group of homeomorphisms of the basilica Julia set.  Each element of this group acts as a piecewise-linear homeomorphism of the unit circle that preserves the invariant lamination for the basilica.  We develop an analogue of tree pair diagrams for this group which we call arc pair diagrams, and we use these diagrams to prove that the group is finitely generated.  We also prove that the group is virtually simple.
\end{abstract}

\maketitle

\section{Introduction}

In the 1960's, Richard J.~Thompson introduced three infinite groups $F$, $T$, and $V$.  Each of these groups consists of piecewise-linear homeomorphisms, with $F$ acting on the unit interval, $T$ acting on the unit circle, and $V$ acting on the Cantor set.  Because of their simple definitions and unique array of properties, these groups have attracted significant attention from geometric group theorists.

One of the puzzling aspects of Thompson's groups is the fact that there are precisely three of them.  Most groups of interest in geometric group theory arise in large infinite families, so it is natural to ask whether the Thompson groups fit into any larger framework.  For this reason, many generalizations of the Thompson groups have been proposed, including Higman's groups~$G_{n,r}$~\cite{higman}, Brown's groups $F_{n,r}$ and~$T_{n,r}$~\cite{brown}, the piecewise-linear groups of Bieri and Strebel~\cite{bieri-strebel} and Stein~\cite{stein}, the braided version of~$V$ studied independently by Brin~\cite{brin} and Dehornoy~\cite{dehornoy}, the two-dimensional version of $V$ described by Brin~\cite{brin2}, and the diagram groups of Guba and Sapir~\cite{guba-sapir}.

The definitions of the three Thompson groups depend heavily on the self-similar structure of the spaces on which these groups act.  For example, each half of the unit interval is similar to the whole interval, as is each quarter or eighth.  For this reason, it is natural to ask whether there are Thompson-like groups associated with other self-similar structures, such as fractals.

In this paper we define a Thompson-like group $T_B$ that acts by homeomorphisms on the basilica Julia set (the Julia set for the quadratic polynomial $z^2-1$).  Each element of this group can be described as a piecewise-linear homeomorphism of the unit circle that preserves the invariant lamination for the basilica, as defined in~\cite{thurston}.  This group is finitely generated, and it possesses an analogue of tree-pair diagrams which we refer to as ``arc pair diagrams''.

For simplicity, we have restricted our investigation to the basilica. We expect that most of our results generalize easily to other Julia sets with similar structure to the basilica, such as the Douady rabbit.  Unfortunately, it is not clear how to best generalize the definition of $T_B$ to quadratic Julia sets whose laminations have a more complicated structure.  We would ultimately like to associate a Thompson group to every quadratic Julia set, in such a way that $T$~is associated with the circle (the Julia set for~$z^2$), $F$~is associated with the line segment (the Julia set for $z^2-2$), and $V$ is associated with disconnected quadratic Julia sets (which are always homeomorphic to the Cantor set).

Finally, we should point out that our group $T_B$ is different from the ``basilica group'' introduced by R.~Grigorchuk and A.~\.{Z}uk \cite{grigorchuk-zuk}.  The latter is the same as the iterated monodromy group \cite{nekrashevych} of the basilica polynomial, and does not act on the Julia set itself.  It is not clear what, if any, relationship exists between the basilica group and our group $T_B$.

This paper is organized as follows.  In Section~2 we briefly review the definition and properties of Thompson's group~$T$.  In Section~3 we give the necessary background on the basilica Julia set and its corresponding invariant lamination.  In Section~4 we define the elements of the group $T_B$ using arc pair diagrams. In Section~5 we prove that $T_B$ forms a group, and we characterize the elements of $T_B$ among all piecewise-linear homeomorphisms.  Section~6 describes an isomorphism between a subgroup of $T_B$ and Thompson's group~$T$.  In Section~7 we show that $T_B$ is finitely generated, and in Section~8 we prove that $T_B$ is virtually simple.

\bigskip
\section{Thompson's Group $T$}
In this section we briefly review the necessary background on Thompson's group~$T$.  See \cite{cannon-floyd-parry} for a more thorough introduction.

Let $S^1$ denote the unit circle, which we identify with $\R/\Z$.  Suppose we cut this circle in half along the points $0$ and $1/2$, as shown in the following figure:
\begin{center}
\includegraphics{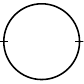}
\end{center}
We then cut each of the resulting intervals in half:
\begin{center}
\includegraphics{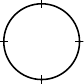}
\end{center}
and then cut some of the new intervals in half:
\begin{center}
\includegraphics{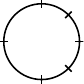}
\end{center}
Continuing in this way, we obtain a subdivision of the circle into finitely many intervals.  Any subdivision obtained in this fashion (by repeatedly cutting in half) is called a \newword{dyadic subdivision}.

The intervals of a dyadic subdivision are all of the form
\[
\left[\frac{k}{2^m},\frac{k+1}{2^m}\right]
\]
for some $k\in\Z$ and $m\in\N$.  Intervals of this form are called \newword{standard dyadic intervals}, and any endpoint of a standard dyadic interval is called a \newword{dyadic point} on the circle.  It is easy to see that any subdivision of the circle into standard dyadic intervals is in fact a dyadic subdivision.

A \newword{dyadic rearrangement} of the circle is any orientation-preserving piecewise-linear homeomorphism $f\colon S^1\to S^1$ that maps linearly between the intervals of two dyadic subdivisions.  Such a homeomorphism can be specified by a pair of dyadic subdivisions, together with a bijection between the intervals that preserves the counterclockwise order.  Three such rearrangements are shown in Figure~\ref{fig:TGenerators}.

The following proposition characterizes dyadic rearrangements.

\begin{proposition}Let $f\colon S^1\to S^1$ be a piecewise-linear homeomorphism.  Then $f$ is a dyadic rearrangement if and only if it satisfies the following conditions:
\begin{enumerate}
\item The derivative of each linear segment of $f$ has the form $2^m$ for some $m\in\Z$.\smallskip
\item Each breakpoint of $f$ maps a dyadic point in the domain to a dyadic point in the range.\qed
\end{enumerate}
\end{proposition}

\begin{figure}[b]
\centering
\fbox{\includegraphics{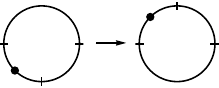}}
\hfill
\fbox{\includegraphics{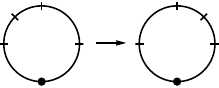}}
\hfill
\fbox{\includegraphics{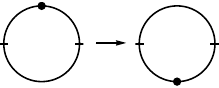}}
\caption{Three dyadic rearrangements of the circle.  We have used a dot in each case to indicate a corresponding pair of intervals in the domain and range.}
\label{fig:TGenerators}
\end{figure}
It is easy to see from this proposition that the dyadic rearrangements of the circle form a group under composition.  This group is known as \newword{Thompson's Group~$\boldsymbol T$}.  The following theorem summarizes some of the basic properties of this group, most of which are proven in \cite{cannon-floyd-parry}.

\bigbreak
\begin{theorem}\quad\label{thm:TFacts}
\begin{enumerate}
\item $T$ is finitely generated.  In particular, $T$ is generated by the three dyadic rearrangements shown in Figure~\ref{fig:TGenerators}.\smallskip
\item $T$ is finitely presented, and indeed has type $F_\infty$ \cite{brown}.\smallskip
\item $T$ is simple.\smallskip
\item $T$ acts transitively on the dyadic points of the circle.\smallskip
\item The stabilizer under $T$ of any dyadic point is isomorphic to Thompson's group~$F$.
\end{enumerate}
\end{theorem}

\bigskip
\section{The Basilica}

In this section, we briefly review the definition of the basilica Julia set and its relation to the basilica lamination. See~\cite{thurston} for further information.

Let $p\colon\C\to\C$ be the polynomial function $p(z)=z^2 - 1$.  If $z \in \C$, the \newword{orbit} of~$z$ under $p$ is the sequence
\[
z,\quad p(z), \quad p(p(z)), \quad p(p(p(z))), \quad \ldots\text{.}
\]
The \newword{filled Julia set} for $p$ is the subset of the complex plane consisting of all points whose orbits remain bounded under~$p$. This set is shown in Figure~\ref{BasilicaPictures}A.  The \newword{Julia~set} for $p$ is the topological boundary of the filled Julia set, as shown in Figure~\ref{BasilicaPictures}B.  This Julia set is known as the \newword{basilica}, and will be denoted by the letter~$B$.
\begin{figure}[tb]
\centering
\subfloat[The filled Julia set]{\fbox{\includegraphics{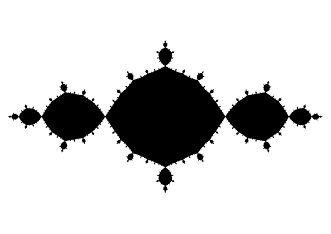}}}
\hfill
\subfloat[The basilica]{\fbox{\includegraphics{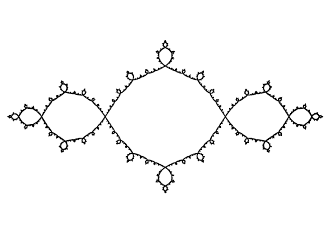}}}
\caption{The filled Julia set and Julia set for $z^2-1$.}
\label{BasilicaPictures}
\end{figure}

The complement of the basilica consists of an exterior region, together with infinitely many interior regions, which we refer to as \newword{components} of the basilica.  Each of these components is a topological disk, and two of these components are said to be \newword{adjacent} if they have a boundary point in common.  Among the components, the most important is the large \newword{central component}, which contains the origin of the complex plane.

\begin{figure}[t]
\centering
\includegraphics{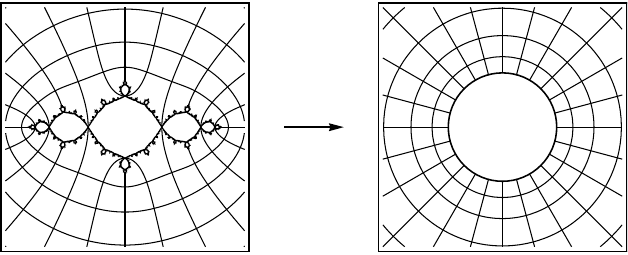}
\caption{The B\"{o}ttcher map $\Phi$ from the exterior of the basilica to the complement of the closed unit disk.}
\label{RiemannMap}
\end{figure}
The exterior region $E$ for the basilica is an open annulus.  By the Riemann Mapping Theorem, this region is conformally equivalent to the complement $\C\backslash D^2$ of the closed unit disk.  Indeed, there exists a unique conformal homeomorphism $\Phi\colon E \to \C\backslash D^2$ making the following diagram commute:
\[
    \xymatrix@R=0.4in@C=0.2in@M=0.5em{
    E \ar[r]^p\ar[d]_{\Phi} & E\ar[d]^{\Phi} \\
    \C\backslash D^2 \ar[r]_{z^2} & \C\backslash D^2
    }
\]
The homeomorphism $\Phi$ is known as the \newword{B\"{o}ttcher map} for the basilica.  A sketch of this map is shown in Figure~\ref{RiemannMap}.

The preimages of radial lines under the B\"{o}ttcher map are known as \newword{external rays}.  Several of these are shown in Figure~\ref{ExternalRays}, with each ray labeled by the corresponding angle.  (By convention, we label angles with real numbers from~$0$ to~$1$.)
\begin{figure}[b]
\centering
\includegraphics{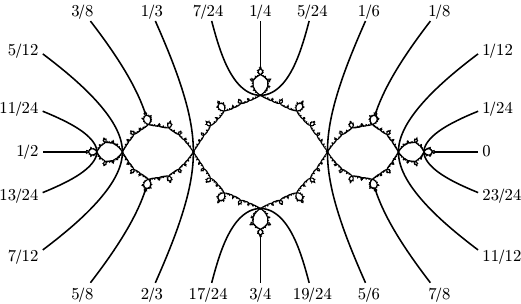}
\caption{External rays for the basilica.  Each ray is labeled by the angle of the corresponding radial segment.}
\label{ExternalRays}
\end{figure}
Each external ray lands at a point on the basilica, which defines a continuous surjection $\psi$ from the unit circle $S^1$ to the basilica~$B$.  This map fits into a commutative diagram:
\[
    \xymatrix@=0.3in@M=0.5em{
    S^1 \ar[r]^{z^2}\ar[d]_{\psi} & S^1\ar[d]^{\psi} \\
    B \ar[r]_p & B
    }
\]
The map $\psi$ is not one-to-one.  For example, the external rays for $1/3$ and $2/3$ end at the same point of the basilica (see Figure~\ref{ExternalRays}), and therefore $\psi(1/3) = \psi(2/3)$.

Since the circle is compact, the map $\psi$ is a quotient map, and therefore the basilica is a quotient of the circle.  The corresponding equivalence relation can be described using an \newword{invariant lamination} (see Figure~\ref{ArcDiagram}), which we refer to as the \newword{basilica lamination}.  This lamination consists of the closed unit disk~$D^2$, together with a hyperbolic arc (or \textit{leaf}) connecting each pair of points on the boundary circle that are identified in the basilica.  For example, there is an arc from $1/3$ to $2/3$ in the basilica lamination because the external rays for these angles land at the same point.
More generally, there is an arc in the lamination between the points
\[
\frac{3k - 1}{3\cdot 2^n} \qquad\text{and}\qquad \frac{3k + 1}{3\cdot 2^n}
\]
for all $k\in\Z$ and $n\geq 0$.  Any homeomorphism of the circle that preserves the equivalence relation defined by these arcs descends to a homeomorphism of the basilica Julia set.
\begin{figure}[tb]
\centering
\includegraphics{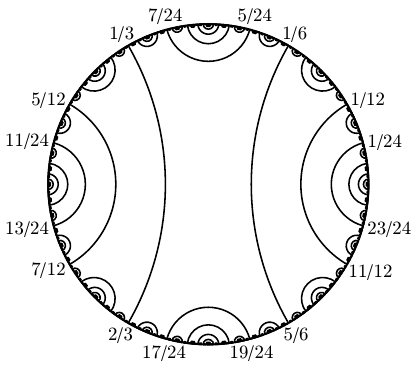}
\caption{The basilica lamination.}
\label{ArcDiagram}
\end{figure}

The complementary components of the basilica lamination are known as \newword{gaps}.  Each of these corresponds to a component of the basilica.  Indeed, the filled Julia set can be described as the quotient of the closed unit disk obtained by contracting each arc of the lamination to a single point.  Note then that two gaps correspond to adjacent components if and only if the closures of the gaps intersect along an arc of the lamination.

\bigskip
\section{Dyadic Rearrangements of the Basilica}
Consider the arcs $\{1/3,2/3\}$ and $\{1/6,5/6\}$ in the unit disk:
\begin{center}
\includegraphics{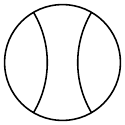}
\end{center}
These arcs divide the unit circle into four subintervals, namely $[1/6,1/3]$, $[1/3,2/3]$, $[2/3,5/6]$,  and $[-1/6, 1/6]$.  Each of these intervals supports its own \newword{primary arc}, which subdivides the interval with ratios $1:2:1$.  We add these arcs to our diagram:
\begin{center}
\includegraphics{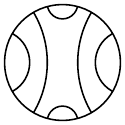}
\end{center}
Together, these six arcs divide the circle into twelve subintervals.  Again, each of these intervals supports a primary arc, and we add some of these arcs to our diagram:
\begin{center}
\includegraphics{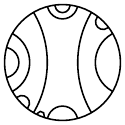}
\end{center}
Continuing in this way, we obtain a finite collection of arcs, which we refer to as an \newword{arc diagram}.  By definition, every arc diagram must include at least the two arcs $\{1/3,2/3\}$ and $\{1/6,5/6\}$.

Any arc diagram divides the unit circle into finitely many intervals, each of which has the form
\[
\left[\frac{3k+1}{3\cdot 2^n},\,\frac{3k+2}{3\cdot 2^n}\right] \qquad\text{or}\qquad \left[\frac{3k-1}{3\cdot 2^{n+1}},\,\frac{3k+1}{3\cdot 2^{n+1}}\right]
\]
for some $k,n\in\{0,1,2,\ldots\}$.  We refer to intervals of this form as \newword{standard intervals}.  It is easy to see that any subdivision of the unit circle into standard intervals can be obtained from some arc diagram.

We can think of an arc diagram as a $2$-complex, with the arcs and intervals as edges and the regions as faces.  We say that two arc diagrams are \newword{isomorphic} if there exists an orientation-preserving isomorphism between the corresponding complexes.

Given an isomorphism $\varphi\colon\mathcal{D}\to\mathcal{R}$ between two arc diagrams, we can define a piecewise-linear homeomorphism $f\colon S^1\to S^1$ by sending each interval of $\mathcal{D}$ linearly to the corresponding interval of $\mathcal{R}$.  Note that such a homeomorphism necessarily preserves the equivalence relation on the circle defined by the basilica lamination, and therefore acts as a self-homeomorphism of the Julia set.  We will refer to such a homeomorphism as a \newword{dyadic rearrangement} of the basilica.  In Section~5 we will prove that the set of dyadic rearrangements of the basilica forms a group.

\begin{figure}[t]
\centering
\includegraphics{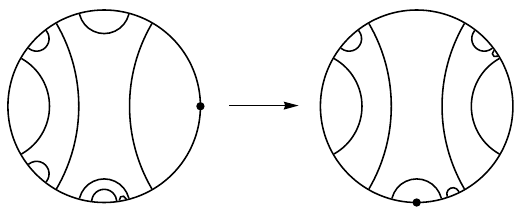}
\caption{An arc pair diagram.}
\label{Isomorphic}
\end{figure}
Figure~\ref{Isomorphic} shows an isomorphism between two arc diagrams, which in turn defines a dyadic rearrangement of the basilica.  We will refer to a picture of this kind as an \newword{arc pair diagram}.  Note that we have marked a corresponding pair of points in the domain and range to specify the isomorphism.

It may seem odd to refer to these homeomorphisms as ``dyadic rearrangements'', since their endpoints are not dyadic fractions.  However, their endpoints have denominators of the form $3\cdot 2^n$, and it is easy to verify that the slopes of a dyadic rearrangement are always powers of two.  Indeed, the dyadic rearrangements of the basilica are all contained in a certain conjugate copy of Thompson's group $T$ acting on the unit circle (see Remark~\ref{rmk:ContainedInT} below).

In Section~7, we will prove that the group of dyadic rearrangements of the basilica is generated by four elements $\alpha$, $\beta$, $\gamma$, and $\delta$.  The following example introduces these elements.

\begin{example}
\label{ex:generators}
\begin{figure}[tbp]
\centering
\fbox{\includeapd{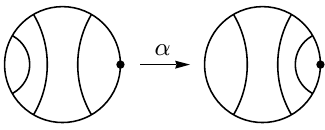}}\qquad\fbox{\includeapd{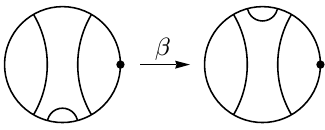}}\\[12pt]
\fbox{\includeapd{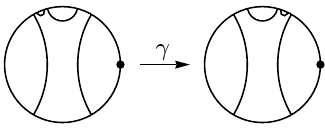}}\qquad\fbox{\includeapd{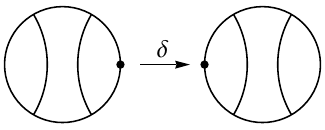}}
\caption{Arc pair diagrams for $\alpha$, $\beta$, $\gamma$, and $\delta$.}
\label{ArcPairGens}
\end{figure}
Let $\alpha$, $\beta$, $\gamma$, and $\delta$ be the dyadic rearrangements defined by the arc pair diagrams in Figure~\ref{ArcPairGens}.  The following table lists the breakpoints for these homeomorphisms:
\[
\renewcommand{\arraystretch}{1.5}
\begin{tabular}{c @{\hspace{.2 in}} c @{\hspace{.2 in}} c}
\hline
Name & Domain Breakpoints & Range Breakpoints\\
\hline
$\alpha$ & 1/3, 2/3 & 1/6, 5/6\\
$\beta$ & 1/6, 2/3, 17/24, 5/6 &  1/6, 7/24, 1/3, 5/6 \\
$\gamma$ & 1/6, 7/24, 29/96, 1/3 & 1/6, 19/96, 5/24, 1/3 \\[3 pt]
\hline
\end{tabular}
\]
The homeomorphism $\delta$ corresponds to the linear function $\theta \mapsto \theta + 1/2$, which has no breakpoints.  The graphs of $\alpha$, $\beta$, and $\gamma$ are shown in Figure~\ref{GraphsGens}.

\begin{figure}[p]
\centering
$\underset{\textstyle\alpha}{\includegraphics{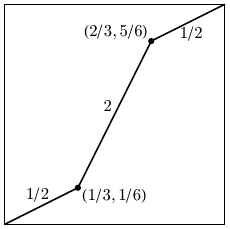}}
\hfill
\underset{\textstyle\beta}{\includegraphics{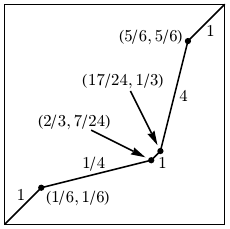}}
\hfill
\underset{\textstyle\gamma}{\includegraphics{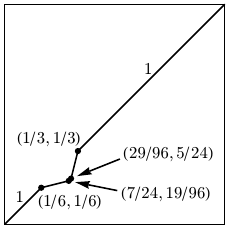}}$
\vspace{-1ex}\caption{Graphs of $\alpha$, $\beta$, and $\gamma$.\\ ~ \\ ~ \\}
\label{GraphsGens}
\end{figure}
\begin{figure}[p]
\centering
\includegraphics{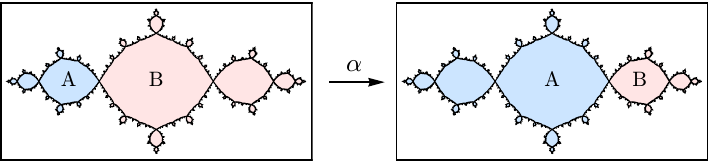} \\[8pt]
\includegraphics{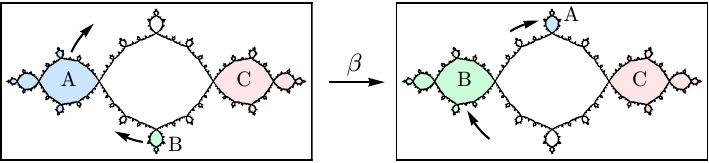} \\[8pt]
\includegraphics{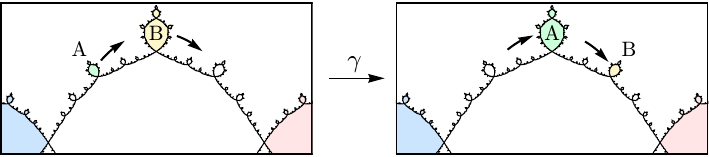} \\[8pt]
\includegraphics{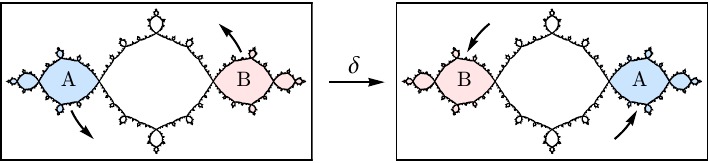}
\caption{The actions of $\alpha$, $\beta$, $\gamma$, and $\delta$ on the basilica.}
\label{ActionOnBasilica}
\end{figure}
Figure \ref{ActionOnBasilica} shows the actions of $\alpha$, $\beta$, $\gamma$ and $\delta$ on the basilica.  Geometrically, $\alpha$~fixes the right and left endpoints of the basilica, and slides each component along the real axis one step to the right.  Both $\beta$ and $\gamma$ permute the components adjacent to central component, stretching and compressing portions of the boundary of the central component as needed.  For example, $\beta$ shrinks the top portion of the basilica that connects C to A in Figure~\ref{ActionOnBasilica}, and expands the lower-right portion that stretches from B to~C, sliding the portion connecting A and~B from the lower left to the top left.  The element $\gamma$ acts in a similar way, but is supported entirely on the top portion of the basilica. Finally, $\delta$ acts as a $180^\circ$ rotation of the basilica around the origin.

Note that $\beta$, $\gamma$, and $\delta$ act on the boundary of the central component in the same way that generators for Thompson's group~$T$ act on the unit circle (see~Figure~\ref{fig:TGenerators}).  Indeed, we will show in Section~\ref{CentCompSect} that $\beta$, $\gamma$, and $\delta$ generate an isomorphic copy of Thompson's group~$T$.
\end{example}

\bigskip
\section{The Group $T_B$}
Let $T_B$ be the set of all dyadic rearrangements of the basilica.  In this section, we characterize the elements of $T_B$ among all piecewise-linear homeomorphisms of the circle.  Using this characterization, we will prove that $T_B$ forms a group under composition.

As we have seen, every element of $T_B$ can be defined using an arc pair diagram.  However, the arc pair diagram for a given dyadic rearrangement is not necessarily unique.

\begin{definition}
An \newword{expansion} of an arc pair diagram is obtained by attaching primary arcs to a corresponding pair of intervals in the domain and range (see Figure~\ref{ReduceExpand}).
The inverse of an expansion is called a \newword{reduction}.  An arc pair diagram is \newword{reduced} if it is not subject to any reductions.
\end{definition}

Expansions and reductions do not change the underlying piecewise-linear homeo\-morphism---they simply correspond to adding or removing unnecessary subdivisions of the domain and range intervals.

\begin{proposition}
Every element of $T_B$ has a unique reduced arc pair diagram.
\end{proposition}
\begin{proof} Let $f \in T_B$.  Given a standard interval $I$, we say that $I$ is  \textit{regular} if $f$ is linear on~$I$, and the image $f(I)$ is also a standard interval.  The domain intervals of any arc pair diagram for $f$ must be regular, and any subdivision of the domain into regular intervals determines an arc pair diagram.  In particular, an arc pair diagram for $f$ is reduced if and only if each of the regular intervals in its domain is maximal under inclusion.

Now, for any pair of standard intervals, either one is contained in the other or the two intervals have disjoint interiors.  It follows that any two maximal regular intervals have disjoint interiors, so there can be only one subdivision of the circle into maximal regular intervals.
\end{proof}

\begin{figure}[tb]
\centering
\includeapd{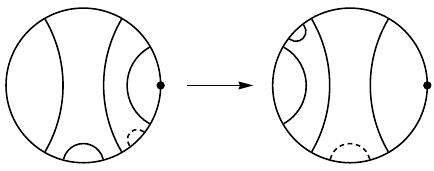}
\caption{An expansion of an arc pair diagram.  The arcs to be added are drawn as dashed lines.}
\label{ReduceExpand}
\end{figure}
Given the reduced arc pair diagrams for two dyadic rearrangements $f,g \in T_B$, there is a simple algorithm to compute an arc pair diagram for the composition $g\circ f$, as illustrated in the following example.

\begin{example}
Suppose we wish to compute an arc pair diagram for the composition $\beta\circ\alpha$, starting with the arc pair diagrams for $\alpha$ and $\beta$ shown in Figure~\ref{ArcPairGens}.  The first step is to expand the arc pair diagram for $\alpha$ until the range diagram for $\alpha$ contains the domain diagram for~$\beta$:
\begin{center}
\includeapd{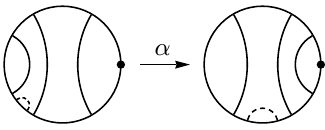}
\quad\quad
\includeapd{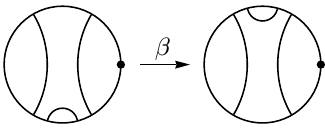}
\end{center}
Next, we expand the arc pair diagram for $\beta$ so that the domain diagram for $\beta$ is the same as the range diagram for $\alpha$:
\begin{center}
\includeapd{CompositionExpandedAlpha}
\quad\quad
\includeapd{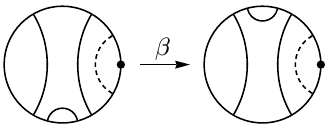}
\end{center}
Finally, we construct an arc pair diagram for $\beta\circ\alpha$ using the domain diagram for $\alpha$ and the range diagram for~$\beta$:
\begin{center}
\includeapd{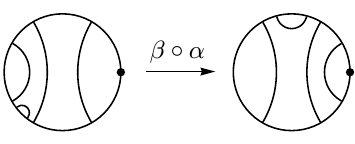}
\end{center}
The isomorphism between these two arc diagrams is obtained by composing the isomorphisms from the expanded arc pair diagrams for $\alpha$ and~$\beta$.
\end{example}

It is possible to use the above algorithm to prove that $T_B$ is a group.  However, we will take a different approach.

\begin{theorem}\label{GroupTB}Let $f$ be a piecewise-linear homeomorphism of the circle.  Then $f$ induces a dyadic rearrangement of the basilica if and only if it satisfies the following requirements:
\begin{enumerate}
\item The equivalence relation on the circle defined by the basilica lamination is invariant under~$f$.\smallskip
\item Every breakpoint of $f$ is the endpoint of an arc in the lamination.
\end{enumerate}
In particular, $T_B$ forms a group under composition.
\end{theorem}
\begin{proof}
The forward direction follows from the definition of a dyadic rearrangement of the basilica.  For the converse, recall that the endpoints of the arcs of the basilica lamination all have the form
\[
\frac{k+1}{3\cdot 2^m} \qquad\text{or}\qquad \frac{k+2}{3\cdot 2^m},\tag*{(1)}
\]
where $k\in\Z$ and $m\in\{0,1,2,\ldots\}$.  Each linear segment of $f$ must preserve this set of points, and must therefore have the form
\[
\theta \;\mapsto\; 2^j \left(\theta + \frac{n}{2^m}\right)\tag*{(2)}
\]
for some $j,m,n\in\Z$.

Now, let $M$ be the maximum value of $m$ that appears in formulas (1) and (2) for the breakpoints and linear segments of~$f$. Let $\mathcal{D}$ be any dyadic subdivision of the unit circle into standard intervals of length less than~$1/2^M$.  Then $f$ must be linear on each interval of $\mathcal{D}$, and each of these intervals maps to a standard interval under~$f$.  The images of the intervals of $\mathcal{D}$ form a dyadic subdivision $\mathcal{R}$ of the range, and $f$ maps linearly between the intervals of $\mathcal{D}$ and~$\mathcal{R}$.  Moreover, since $f$ preserves the arcs of the basilica lamination, $f$~must induce an isomorphism between the arc diagrams for $\mathcal{D}$ and~$\mathcal{R}$, and therefore $f$ is a dyadic rearrangement.
\end{proof}

\begin{remark}\label{rmk:ContainedInT}As discussed in the proof of this theorem, each linear segment of an element of~$T_B$ has the form
\[
\theta \;\mapsto\; 2^n \theta + d
\]
where $n\in\Z$ and $d$ is a dyadic rational.  However, the breakpoints of such an element are not dyadic rationals, so elements of $T_B$ do not act in the circle in the same way as elements of Thompson's group~$T$.

On the other hand, the group $T_B$ is isomorphic to a subgroup of $T$. In particular, consider the group $T(3)$ of all piecewise-linear homeomorphisms of the circle satisfying the following conditions:
\begin{enumerate}
\item The derivative of each linear segment has the form $2^m$ for some $m\in\Z$.
\item The coordinates of each breakpoint have the form $\dfrac{k}{3\cdot 2^n}$ for some $k,n\in\Z$.
\end{enumerate}
Then $T(3)$ is isomorphic to Thompson's group~$T$, since it is conjugate to $T$ in the group of all piecewise-linear homeomorphisms of the circle.  Each element of~$T_B$ acts on the circle as an element of $T(3)$, and therefore $T_B$ is isomorphic to a subgroup of Thompson's group~$T$.

More explicitly, $T_B$ is isomorphic to the subgroup of $T$ consisting of all elements that preserve the equivalence relation on the dyadics corresponding to the lamination shown in Figure~\ref{DyadicLamination}.
\begin{figure}[t]
\centering
\includegraphics{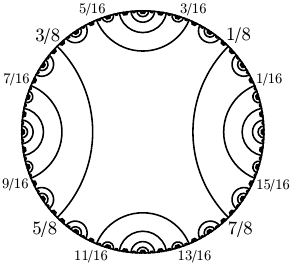}
\caption{A lamination whose arcs have dyadic endpoints.  The corresponding subgroup of $T$ is isomorphic to~$T_B$.}
\label{DyadicLamination}
\end{figure}
\end{remark}

\bigskip
\section{The Central Component}\label{CentCompSect}

\begin{figure}[tb]
\centering
\subfloat[]{\includegraphics{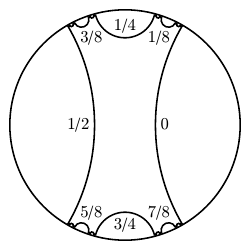}}
\qquad\qquad
\subfloat[]{\includegraphics{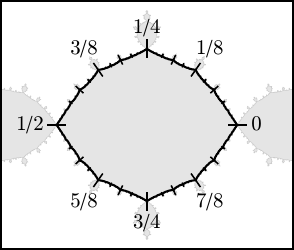}}
\caption{(A) Dyadic labels for the central arcs of the basilica lamination. (B) The corresponding dyadic points on~$\partial C$.}
\label{fig:CentralComponentPictures}
\end{figure}
Consider the central component $C$ of the basilica Julia set.  This component is a quotient of the \newword{central gap} in the basilica lamination, i.e.~the gap containing the center point of the disk.  In particular, $C$ is obtained from the central gap by collapsing each of the infinitely many \newword{central arcs} that surround this gap.  We can label these central arcs using the dyadic rationals, as shown in Figure ~\ref{fig:CentralComponentPictures}A.

Each central arc corresponds to a single point in the boundary $\partial C$ of~$C$ the central component.  These points are dense in $\partial C$, as shown in Figure~\ref{fig:CentralComponentPictures}B.  Therefore, the labeling of the central arcs by dyadics induces a well-defined homeomorphism from $\partial C$ to the unit circle.  In particular, we obtain a canonical action of Thompson's group $T$ on~$\partial C$

\begin{remark}It is possible to use the Riemann Mapping Theorem to define the homeomorphism $\partial C \to S^1$ directly, without reference to the invariant lamination.  In particular, consider the Riemann map from $C$ to the unit disk that sends $0$ to $0$ with derivative $1$, as shown in Figure~\ref{fig:CenterBulbRiemann}.  This map extends to the boundary, and the resulting homeomorphism from $\partial C$ to the circle is the same as the one defined above.
\end{remark}

As we shall see, the action of Thompson's group $T$ on $\partial C$ is closely related to the action of $T_B$ on the basilica.  In particular, consider the following two subgroups of~$T_B$:

\begin{definition}\quad
\begin{enumerate}
\item The \newword{stabilizer} of $C$, denoted $\stab(C)$, is the group of all elements of $T_B$ that map the central component $C$ to itself.\smallskip
\item The \newword{rigid stabilizer} of $C$, denoted $\rist(C)$, is the group of all elements of $T_B$ that have an arc pair diagram consisting solely of central arcs.
\end{enumerate}
\end{definition}
\begin{figure}[t]
\centering
\includegraphics{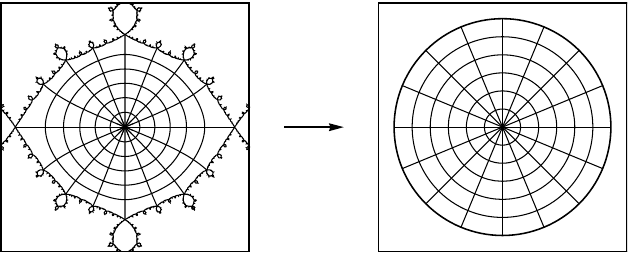}
\caption{A Riemann map on the central component.}
\label{fig:CenterBulbRiemann}
\end{figure}

Every element of the stabilizer maps $C$ to itself, and therefore the group $\stab(C)$ acts on $\partial C$ by homeomorphisms.  The rigid stabilizer $\rist(C)$ is a subgroup of $\stab(C)$.  Though we have defined it using arc pair diagrams, $\rist(C)$ can also be characterized as the group all elements of $\stab(C)$ that extend conformally to the interiors of the non-central components of~$B$. It is ``rigid'' in the sense that an element of $\rist(C)$ is determined entirely by its action on~$\partial C$.

The following theorem describes the action of these groups on the boundary of the central component.

\begin{theorem}\label{thm:MapsToT}
Each element of\/ $\stab(C)$ acts on\/ $\partial C$ as an element of Thompson's group~$T$.  In particular, $\rist(C)$ acts on $\partial C$ as an isomorphic copy of $T$.
\end{theorem}
\begin{proof}Consider first an element $f\in\rist(C)$.  For example, $f$ could be the following element:
\begin{center}
\includeapd{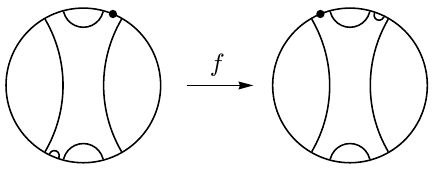}
\end{center}
Note that $f$ must have the same number of central arcs in its domain and range diagrams.  These arcs correspond precisely to the cut points for two dyadic subdivisions of the circle:
\begin{center}
\includeapd{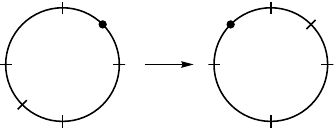}
\end{center}
This correspondence clearly defines an isomorphism $\tau\colon \rist(C) \to T$.

Next, observe that the action of an element $f\in\rist(C)$ on the central arcs on the basilica lamination is precisely the same as the action of $\tau(f)$ on the dyadic points of the circle.  Since the images of the central arcs are dense in $\partial C$, it follows that $f$ acts on $\partial C$ in exactly the same way as~$\tau(f)$.

Finally, observe that if $g$ is any element of $\stab(C)$, then $g$ acts on~$\partial C$ in exactly the same way as some element $f\in \rist(C)$.  Specifically, an arc pair diagram for $f$ can be obtained by removing all of the non-central arcs from an arc pair diagram for $g$.  For example, the following element of $\stab(C)$ acts on $\partial C$ in precisely the same way as the element $f\in \rist(C)$ shown above:
\begin{center}
\hfill\includeapd{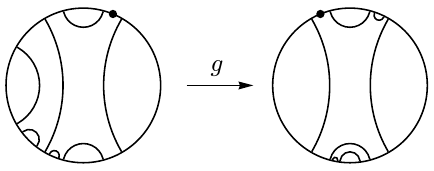}\hfill\qedhere
\end{center}
\end{proof}

\begin{corollary}\label{cor:GensForT}The group $\rist(C)$ is generated by the elements $\beta$, $\gamma$, and $\delta$ defined in Example~\ref{ex:generators}
\end{corollary}
\begin{proof}These three elements act on $\partial C$ in the same way as the three generators for Thompson's group~$T$ shown in Figure~\ref{fig:TGenerators}.
\end{proof}

\begin{remark}Theorem~\ref{thm:MapsToT} shows that $T_B$ contains an isomorphic copy of Thompson's group~$T$.  This has a few basic consequences:
\begin{enumerate}
\item $T_B$ contains a non-abelian free group of rank two.\smallskip
\item $T_B$ contains a free abelian group of infinite rank.
\end{enumerate}
It follows from (1) that $T_B$ has exponential growth, and is not amenable.
\end{remark}

\begin{remark}One result of Theorem~\ref{thm:MapsToT} is that the group $\stab(C)$ has the structure of a semidirect product:
\[
\stab(C) \;=\; \mathrm{fix}(C) \rtimes T.
\]
Here $\mathrm{fix}(C)$ denotes the kernel of the homomorphism $\stab(C) \to T$, i.e.~the group of all elements of $\stab(C)$ that restrict to the identity on~$\partial C$.

Though we shall not prove it here, $\stab(C)$ is actually isomorphic to a direct limit of iterated wreath products:
\[
T \quad\longrightarrow\quad F \wr T \quad\longrightarrow\quad F \wr F \wr T \quad\longrightarrow\quad \cdots
\]
where $T$ acts on the dyadic points of the circle, and $F$ denotes Thompson's group~$F$ acting on the dyadic points in the interval $(0,1)$.  Note that each of the wreath products above is a restricted wreath product, i.e.~a semidirect product whose first component is an infinite direct \textit{sum} of groups.
\end{remark}

\bigskip
\section{Generators}

In this section we prove the following theorem.

\begin{theorem}\label{thm:generators}
The group $T_B$ is generated by the elements $\{\alpha,\beta,\gamma,\delta\}$ defined in Example~\ref{ex:generators}.
\end{theorem}

We will need the following lemma.

\begin{lemma}\label{lemma:TransitiveOnRegions}
The group $\langle\alpha,\beta,\gamma,\delta\rangle$ acts transitively on the gaps of the basilica lamination.
\end{lemma}
\begin{proof}
Define the \textit{depth} of each gap in the basilica lamination to be the number of arcs separating it from the central gap.  Thus the central gap has depth zero, adjacent gaps have depth one, and so forth.

Let $R_n$ be any gap of depth $n$.  We will show that $R_n$ can be mapped to the central gap using an element of $\langle\alpha,\beta,\gamma,\delta\rangle$.  We proceed by induction on $n$.

If $n=0$ then $R_n$ is already the central gap and we are done.  Otherwise, $R_n$ is connected to the central gap $R_0$ through a path $R_0,R_1,\ldots,R_n$ of pairwise adjacent gaps.

Consider the gap $R_1$, which borders the central gap along a central arc~$A$.  Now, recall from Section~6 that the group $\langle\beta,\gamma,\delta\rangle$ acts on the central arcs in the same way that Thompson's group~$T$ acts on the dyadic points of the unit circle.  In particular, the group $\langle\beta,\gamma,\delta\rangle$ acts transitively on the central arcs, so there exists an element $f\in \langle\beta,\gamma,\delta\rangle$ for which $f(A)$ is the arc $\{1/3,2/3\}$.  Then $f(R_1)$ must be the gap directly to the left of the central gap, so $(\alpha\circ f)(R_1)$ is the central gap.  Then $(\alpha\circ f)(R_1), \ldots, (\alpha\circ f)(R_n)$ is a path of pairwise adjacent gaps, where $(\alpha\circ f)(R_1)$ is the central gap, and therefore $(\alpha\circ f)(R_n)$ has depth $n-1$.  By our induction hypothesis, it follows that $(\alpha\circ f)(R_n)$ can be mapped to the central gap, and therefore $R_n$ can as well.
\end{proof}

\begin{proof}[Proof of Theorem \ref{thm:generators}]
Let $f\in T_B$.  By the lemma, we may assume that $f$ maps the central gap to itself, i.e.~that $f\in\stab(C)$.  We must show that $f\in\langle\alpha,\beta,\gamma,\delta\rangle$.  Let $n$ be the total number of arcs in the reduced arc pair diagram for~$f$. We proceed by induction on~$n$.

Note that $n\geq 4$, since the arcs $\{1/3,2/3\}$ and $\{1/6,5/6\}$ must appear in both the domain and range diagrams.  The base case is $n=4$, for which $f$ must be either the identity or the element $\delta$, both of which lie in $\langle\alpha,\beta,\gamma,\delta\rangle$.

Suppose that $n\geq 5$. Since $f\in \stab(C)$, the action of $f$ permutes the central arcs of the lamination, so any arc pair diagram for $f$ must map central arcs to central arcs.  For example, the reduced arc pair diagram for $f$ may look like
\begin{center}
\includeapd{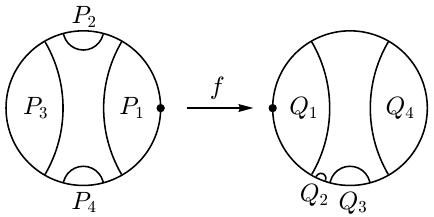}
\end{center}
where $P_1,\ldots,P_4$ and $Q_1,\ldots,Q_4$ denote the arcs in the indicated sections of the diagram.

Now, the permutation of the central arcs is determined by the action of $f$ on the boundary of the central component.  By the results of the previous section, there exists an element $g\in\langle\beta,\gamma,\delta\rangle$ that permutes the central arcs in the same way:
\begin{center}
\includeapd{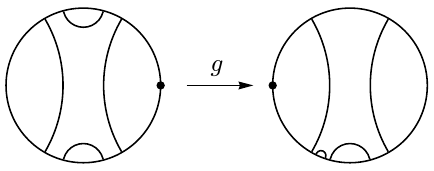}
\end{center}
Then the composition $h = g^{-1}\circ f$ fixes each central arc:
\begin{center}
\includeapd{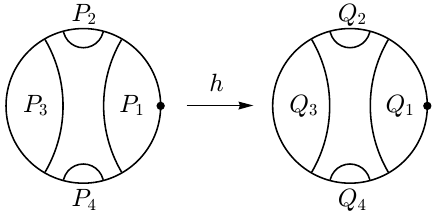}
\end{center}
Note that $h$ has at most $n$ arcs in its reduced arc pair diagram.  It suffices to prove that $h \in \langle\alpha,\beta,\gamma,\delta\rangle$.  We shall prove this using two cases.

\medskip\noindent \textit{Case 1: $h$ has non-central arcs in more than one section.} \\
Let $A_1,\ldots,A_m$ be the central arcs in the arc pair diagram for $h$, and suppose that $h$ has non-central arcs behind more than one~$A_i$.  Then we can express $h$ as a composition $h_1\circ \cdots \circ h_m$, where each $h_i$ has the same central arcs as $h$, but has non-central arcs only in the section of the diagram bounded by~$A_i$.  For example, the element $h_2$ might have the following arc pair diagram:
\begin{center}
\includeapd{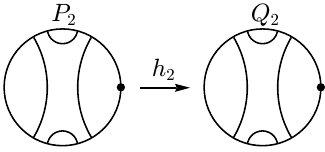}
\end{center}
Then each $h_i$ must have fewer than $n$ arcs.  By induction, it follows that each $h_i \in \langle\alpha,\beta,\gamma,\delta\rangle$, and therefore $h\in\langle\alpha,\beta,\gamma,\delta\rangle$ as well.

\medskip\noindent \textit{Case 2: $h$ has non-central arcs in only one section.} \\
Suppose that all of the non-central arcs of $h$ lie behind a single central arc~$A$, and let $r$ be the element of $\langle\beta,\gamma,\delta\rangle$ that cyclically permutes the central arcs of $h$:
\begin{center}
\includeapd{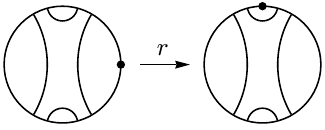}
\end{center}
Then $r^j(A) = \{1/6,5/6\}$ for some exponent~$j$.  Let $k = r^j\circ h \circ r^{-j}$.  Note that we can conjugate $h$ by $r^j$ without expanding~$h$, so $k$ has at most $n$~arcs.  Moreover, all of the non-central arcs of~$k$ lie to the right of $\{1/6,5/6\}$.  It suffices to prove that $k\in\langle\alpha,\beta,\gamma,\delta\rangle$.

Consider the reduced arc pair diagram for $k$:
\begin{center}
\includeapd{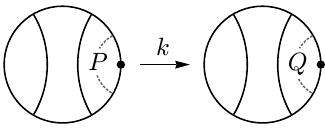}
\end{center}
If $k$ is the identity, we are done.  Otherwise, both $P$ and $Q$ must contain the arc $\{1/12,5/12\}$, as indicated in the picture.  In this case, we can conjugate $k$ by $\alpha$ without performing any expansions of~$k$.  The resulting arc pair diagram for $\alpha^{-1}\circ k\circ \alpha$ has the form:
\begin{center}
\includeapd{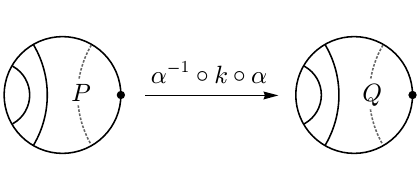}
\end{center}
We can now cancel the arcs $\{5/12,7/12\}$ in the domain and range to yield an arc pair diagram for $\alpha^{-1}\circ k \circ \alpha$ with fewer than $n$~arcs.  By our induction hypothesis, it follows that $k\in\langle\alpha,\beta,\gamma,\delta\rangle$.
\end{proof}

\begin{remark}
Now that we have shown that $T_B$ is finitely generated, we would like to know whether it is finitely presented.  Though $F$, $T$, and $V$ are finitely presented, it seems that the proof for these groups does not generalize to $T_B$.  As a result, we suspect that $T_B$ is not finitely presented.
\end{remark}

\bigskip
\section{The Commutator Subgroup}

\begin{figure}[tb]
\centering
\includegraphics{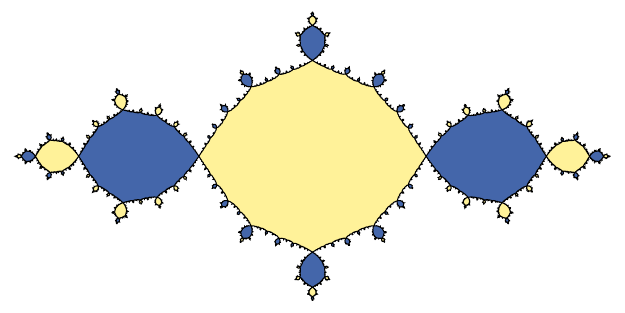}
\caption{The two-coloring of the components of the basilica.}
\label{fig:TwoColoring}
\end{figure}
In this section we investigate the commutator subgroup $[T_B,T_B]$ of $T_B$.  First, consider the two-coloring of the components of the basilica shown in Figure~\ref{fig:TwoColoring}.  This two-coloring is related to dynamics: if $p(z) = z^2 -1$, then $p\circ p$ has attracting fixed points at $0$ and $-1$, and the colors represent the basins of attraction for these fixed points.

\begin{theorem}The commutator subgroup $[T_B,T_B]$ is the index-two subgroup of\/ $T_B$ consisting of all elements that preserve the two-coloring of the components of the basilica.  This group is generated by the elements $\{\beta,\gamma,\delta,\beta^\alpha,\gamma^\alpha,\delta^\alpha\}$.
\end{theorem}
\begin{proof}Define a homomorphism $\pi\colon T_B \to \Z/2\Z$ by $\pi(f) = 0$ if $f$ preserves the coloring of the components, and $\pi(f) = 1$ if $f$ switches the two colors.  Since $\Z/2\Z$ is abelian, we know that $[T_B,T_B] \leq \ker(\pi)$.  We would like to show that $\ker(\pi) \leq [T_B,T_B]$.

By Schreier's Lemma, $\ker(\pi)$ is generated by the six elements in the statement of the theorem together with the element $\alpha^2$.   It is easy to check that $\alpha^2 = \delta^{-1}\circ\delta^\alpha$, which proves that $\ker(\pi)$ is generated by the six given elements.  We must show that each of these elements lies in the commutator subgroup.

From Section~6, we know that $\beta$, $\gamma$, and $\delta$ generate a subgroup of~$T_B$ isomorphic to Thompson's group~$T$.  Since $T$ is simple, $[T,T] = T$, so $\beta$,~$\gamma$, and $\delta$ must be products of commutators in $\langle\beta,\gamma,\delta\rangle$.  Hence $\beta,\gamma,\delta\in[T_B,T_B]$.  Since $[T_B,T_B]$ is normal in $T_B$, it follows that $\beta^\alpha,\gamma^\alpha,\delta^\alpha\in[T_B,T_B]$, so $\ker(\pi) = [T_B,T_B]$.
\end{proof}

\begin{corollary}The abelianization of $T_B$ is isomorphic to $\Z/2\Z$.
\end{corollary}

\begin{remark} It is not hard to show that the epimorphism $T_B \to \Z/2\Z$ splits.  For example, consider the element $\delta\circ\alpha$, which acts as an order-two ``rotation'' of the basilica around the point where the central component and the main left component touch.  This element switches the two colors, and therefore defines a splitting $\Z/2\Z \to T_B$. It follows that $T_B$ is a semidirect product of $[T_B,T_B]$ with~$\Z/2\Z$.
\end{remark}

\begin{theorem}The group $[T_B,T_B]$ is simple.
\end{theorem}
\begin{proof}This proof follows the basic outline for the proof of the simplicity of~$T$, which in turn is based on the work of Epstein~\cite{epstein} on the simplicity of groups of diffeomorphisms.  In particular, we use Epstein's double commutator trick, together with an argument that $T_B$ is generated by elements of small support.

Let $N$ be a nontrivial normal subgroup of $[T_B,T_B]$.  We wish to prove that $N = [T_B,T_B]$.

Let $f$ be a non-trivial element of $N$.  Since the dyadic points are dense on the unit circle, there must exist some dyadic rational $q$ that is not fixed by~$f$.  Then there exists a sufficiently small standard interval $I$ containing $q$ such that
\begin{itemize}
\item The endpoints of $I$ are joined by an arc in the basilica lamination, and\smallskip
\item The image interval $J = f(I)$ is disjoint from $I$.
\end{itemize}
Now, consider any pair of elements $g,h\in [T_B,T_B]$ with support in $I$.  Since $f(I)=J$, the conjugate $f\circ g^{-1}\circ f^{-1}$ has support in~$J$.  Then the commutator $[g,f] = g\circ f\circ g^{-1}\circ f^{-1}$ has support in $I\cup J$, and agrees with $g$ on~$I$.  Since $h$ has support in $I$, it follows that
\[
\bigl[\,[g,f],h\,\bigr] \;=\; [g,h].
\]
The double commutator on the left must be an element of~$N$, and therefore $[g,h]\in N$ for every pair of elements $g,h\in [T_B,T_B]$ with support in~$I$.

Now, since each arc of the basilica lamination borders gaps of two different colors, the group $[T_B,T_B]$ acts transitively on these arcs. In particular, there exists an element $k\in[T_B,T_B]$ that maps the arc joining the endpoints of $I$ to $\{1/3,2/3\}$.  Then $k(I)$ must be either $[1/3,2/3]$ or $[-1/3,1/3]$.  Replacing $I$ with a smaller interval if necessary, we may assume that $k(I) = [-1/3,1/3]$.  Conjugating the result of the previous paragraph by $k$, we find that $[g,h] \in N$ for any $g,h\in [T_B,T_B]$ with support on the interval $[-1/3,1/3]$.

Next, recall from Section 6 that the group $\rist(C)$ is isomorphic to Thompson's group~$T$.  Under this isomorphism, the elements of $\rist(C)$ with support on $[-1/3,1/3]$ correspond to the stabilizer of~$1/2$ in~$T$, which is a copy of Thompson's group~$F$.  Since $F$ is not abelian, there exists at least one pair of elements $g,h\in\rist(C)$ with support in $[-1/3,1/3]$ for which $[g,h]$ is nontrivial. Then $[g,h]$ lies in both $N$ and $\rist(C)$, so the intersection $N\cap \rist(C)$ is a nontrivial normal subgroup of $\rist(C)$.  But $\rist(C) \cong T$ and $T$ is simple, so $N\cap\rist(C) = \rist(C)$, and therefore $\rist(C) \leq N$.

Now consider the conjugate subgroup $\rist(C)^\alpha$, which can be interpreted as the rigid stabilizer of the component immediately to the left of the central component.  If we choose $I$ so that $k(I) = [1/3,2/3]$, we can use the same argument as before to show that $N\cap \rist(C)^\alpha$ has a nontrivial element, and therefore $\rist(C)^\alpha \leq N$.

So far we have proven that $\rist(C)\cup\rist(C)^\alpha \subseteq N$.  However, the group $\rist(C)$ is generated by $\beta$, $\gamma$, and $\delta$, while $\rist(C)^\alpha$ is generated by $\beta^\alpha$, $\gamma^\alpha$, and $\delta^\alpha$.  Since these six elements generate $[T_B,T_B]$, we conclude that $N$ contains all of $[T_B,T_B]$, and therefore $[T_B,T_B]$ is simple.
\end{proof}

Note that this proof does not actually require that the subgroup $N$ be contained in $[T_B,T_B]$.  Indeed, this proof shows that $[T_B,T_B]$ is the only nontrivial proper subgroup of $T_B$ that is normalized by $[T_B,T_B]$.  In particular:

\begin{corollary}The commutator subgroup $[T_B,T_B]$ is the only nontrivial proper normal subgroup of $T_B$.
\end{corollary}

Since any finite-index subgroup of $T_B$ must contain a normal subgroup of finite index, we also obtain:

\begin{corollary}The commutator subgroup $[T_B,T_B]$ is the only proper finite-index subgroup of~$T_B$.
\end{corollary}

\begin{remark}Although the group $[T_B,T_B]$ is finitely generated and simple, it is not isomorphic to Thompson's group~$T$.  To see this, consider the involutions $\delta$ and $\delta^\alpha$ in $[T_B,T_B]$.  The first is the order-two rotation of the basilica around the origin, while the second is an order-two ``rotation'' of the basilica around the main left component.  Since $\delta$ fixes a yellow component of the basilica (the central component), and $\delta^\alpha$ fixes a blue component (the main left component), these elements cannot be conjugate in $[T_B,T_B]$.  However, Thompson's group~$T$ has only one conjugacy class of involutions, namely the conjugacy class of the order-two rotation of the circle.  (This follows from the results in~\cite{belk-matucci}. See also~\cite{fossas}.)  Thus $[T_B,T_B]$ cannot be isomorphic to~$T$.
\end{remark}

\bigskip
\bibliographystyle{plain}

\end{document}